\documentclass[reqno,a4paper]{amsart}
\usepackage{amssymb,amsmath,amsthm,amstext,amsfonts}
\usepackage[colorlinks=true,citecolor=black,linkcolor=black]{hyperref}
\theoremstyle{plain}
\newtheorem{theorem}{Theorem } 
\newtheorem{prop}[theorem]{Proposition}
\newtheorem{cor}[theorem]{Corollary}
\newtheorem{lemma}[theorem]{Lemma}
\theoremstyle{remark}
\newtheorem{rmk}[theorem]{Remark}
\newtheorem{defn}[theorem]{Definition}

\newcommand{\cB}{\mathcal{B}}
\newcommand{\cC}{\mathcal{C}}
\newcommand{\cH}{{\mathcal H}}
\newcommand{\cL}{\mathcal{L}}
\newcommand{\cU}{\mathcal{U}}
\newcommand{\diam}{{\rm diam}}
\newcommand{\Lip}{{\rm Lip}}

\newcommand{\R}{{\mathbb R}}
\newcommand{\bN}{{\mathbb N}}

\newcommand{\norm}[1]{\left\| #1 \right\|}
\newcommand{\abs}[1]{\left| #1 \right|}

\begin{document}

\title[Disintegration of Hyperbolic Skew Products]{Disintegration of Invariant Measures for Hyperbolic Skew Products}
\keywords{Hyperbolic skew product, Mixing rates, SRB measures.}
\subjclass[2010]{Primary: 37A25;   Secondary:  37C30}
\thanks{O.B. was supported by the Austrian Science Fund, Lise Meitner position M1583.
I.M. was supported in part by European Advanced Grant
StochExtHomog (ERC AdG 320977).
We are grateful to Vitor Ara{\'u}jo and Paulo Varandas for helpful comments.
We are especially grateful to the referee for pointing out numerous statements that required correction and/or clarification.}

\author{Oliver Butterley}
\address{Fakult\"at f\"ur Mathematik, Universit\"at Wien,
Oskar-Morgenstern-Platz 1, 1090 Wien, Austria.
Current address: ICTP, Strada Costiera, 11, I-34151 Trieste, Italy.} 
\email{oliver.butterley@ictp.it}

\author{Ian Melbourne}
\address{Mathematics Institute, University of Warwick, Coventry, CV4 7AL, UK.}
\email{i.melbourne@warwick.ac.uk}
\date{\today}

\begin{abstract}
We study hyperbolic skew products and the disintegration  of the SRB measure into measures supported on local stable manifolds.
Such a disintegration gives a method for passing from an observable $v$ on the skew product to an observable $\bar v$ on the system quotiented along stable manifolds.
Under mild assumptions on the system we prove that the disintegration preserves the smoothness of $v$, firstly in the case where $v$ is H\"older and secondly in the case where $v$ is $\cC^{1}$.
\end{abstract}

\maketitle
\thispagestyle{empty}


\section{Introduction}

We suppose throughout that
$\widehat F:\widehat\Delta\to\widehat\Delta$ has the form of a skew product map
so $\widehat\Delta=\Delta\times N$ are compact metric spaces and
\[
\widehat F(x,z)=(Fx,G(x,z))
\]
 where
$F:\Delta\to\Delta$ and $G:\Delta\times N\to N$ are continuous.
Moreover, we suppose that $\nu$ is an $F$-invariant Borel probability
measure on $\Delta$.
Let $\pi:\widehat\Delta\to\Delta$ be the projection $\pi(x,z)=x$ and note that
$\pi$ defines a semiconjugacy between $\widehat F$ and $F$, i.e.,  $F\circ \pi = \pi \circ \widehat F$.
In different sections of this note we will require the system to satisfy different degrees of regularity but the basic setting is for $F$ to be a uniformly expanding map and for $\widehat F$ to be uniformly contracting in the fibre direction in the sense that 
\begin{equation}
\label{eq:contraction}
\diam\,\widehat F^n\pi^{-1}(x)\to0 \quad \text{as $n\to\infty$, uniformly in $x$}.
\end{equation}

In order to study statistical properties of $\widehat{F}$ it is often convenient to study the statistical properties of the expanding map $F$ and then use this to deduce the behaviour for the hyperbolic map $\widehat{F}$.
This involves associating observables on $\widehat{\Delta}$ to observables on $\Delta$ and the consideration of the possible loss of regularity involved in this process. 
In the symbolic setting, this corresponds to the argument where one-sided observables can be used to approximate two-sided observables (see, for example~\cite[\S1.A]{Bowen75}).

Here we pursue a different approach, inspired by~\cite{AGY06}.
Suppose that $\nu$ is an $F$-invariant probability measure 
on $\Delta$.   A standard construction (see Section~\ref{sec:eta}) yields an $\widehat F$-invariant probability measure $\eta$ on $\widehat\Delta$ such that $\nu = \pi_{*} \eta$. 
We are interested in a disintegration $\smash{{\{\eta_{x}\}}_{x\in \Delta}}$ of $\eta$ in the sense  that each $\eta_{x}$ is a Borel probability measure on $\widehat{\Delta}$ supported on $\pi^{-1}(x)$ and 
\[
\eta(v) ={ \int_{\Delta} \eta_{x}(v) \ d\nu(x)}
\]
 for all continuous $\smash{v:\widehat{\Delta} \to \R}$.
Then, in a natural way, $x\mapsto \bar v=\eta_{x}(v)$ is the observable on $\Delta$ associated to the observable $\smash{v:\widehat{\Delta} \to \R}$. 

The existence (and uniqueness) of such a disintegration goes back to Rohlin~\cite{Rohlin52} in a rather more general context.
The purpose of this note is to study the regularity of the disintegration ${\{\eta_{x}\}}_{x\in \Delta}$ in the sense that the regularity of $v$ is inherited
by $\bar v$.
Such questions are important when studying rates of mixing for hyperbolic systems. 
For example, in~\cite{AGY06} exponential mixing is first proved for a hyperbolic semiflow and then lifted  to the hyperbolic flow using a regularity result for the disintegration as described above. 
In their setting the measure $\eta$ has a smooth density and so the regularity of the disintegration is immediate~\cite[Lemma 4.3]{AGY06}. 
We consider the case where $\nu$ is absolutely continuous, but make no such assumption on $\eta$.
At first glance the situation appears rather bad since in general the invariant measure $\eta$ could be singular along the stable manifolds. 
This turns out not to be a problem and good regularity of the disintegration is still possible in these situations. 
Such a result, in the case when the invariant density is singular along stable manifolds, is required in~\cite{ABV14,AMV14,AV12}. 
In the situations studied in those references, there is a $\cC^k$ global stable foliation where $k=1+\alpha$ or $k=2$.
After a $\cC^k$ change of coordinates, we obtain a skew product map $\hat F$ such
that $F$ and $G$ are $\cC^k$, and  our main results exploit this information.
In general such good regularity of the stable foliation cannot be expected but in many cases, for instance under domination conditions or in low dimensions, the regularity is good (see, for example \cite{ABV14, AV12}).

In Section~\ref{sec:eta}, we recall the argument for the existence of the invariant measure for the hyperbolic system.
Then, in Section~\ref{sec:disint}, we present the construction of the disintegration along stable manifolds.
These sections do not require specific assumptions on $F$, $\nu$, or the rate of contraction 
in~\eqref{eq:contraction}.

Sections~\ref{sec:holder} and~\ref{sec:smooth} contain our main results on
the regularity of the disintegration,
firstly for the  H\"older case and secondly for the $\cC^{1}$ case. 
To prove these results, we require additional regularity assumptions on $F$ and $G$,
absolute continuity of $\nu$ and exponential contraction 
in~\eqref{eq:contraction}.

\section{Invariant Measure on $\widehat\Delta$}
\label{sec:eta} 

In this section, we recall the standard argument for
constructing an invariant measure $\eta$ for $\widehat F:\widehat \Delta\to \widehat \Delta$ (see for example~\cite[Section~6]{APPV09}).
This construction makes use of the invariant measure $\nu$ for $F$ together with
the contracting stable foliation, but the details of the map $F:\Delta\to\Delta$ and the rate of contraction are not required.

\begin{prop} 
 \label{prop:mux}
	Given $v:\widehat{\Delta}\to\R$, define $v_+,v_-:\Delta\to\R$ by setting
$v_+(x)=\sup_{z}v(x,z)$,
$v_-(x)=\inf_{z}v(x,z)$.
Then the limits
\[
\lim_{n\to\infty}\int_{\Delta} {(v\circ \widehat F^n)}_+\,d\nu\quad\text{and}\quad
\lim_{n\to\infty}\int_{\Delta} {(v\circ \widehat F^n)}_-\,d\nu
\]
exist and coincide for all $v$ continuous.  Denote the common limit by $\eta(v)$.
This defines an $\widehat{F}$-invariant probability measure $\eta$ on $\widehat{\Delta}$ and $\pi_*\eta=\nu$.
\end{prop}

\begin{proof}
	Let $v_n^{\pm}=\int_\Delta {(v\circ \widehat F^n)}_{\pm}\,d\nu$.
We have
\begin{align*} 
	{(v\circ \widehat F^{n+1})}_+(x) & =\sup_z v\circ \widehat F^{n+1}(x,z)
=\sup_z v\circ \widehat F^n(Fx,G(x,z))
\\ & \le \sup_z v\circ \widehat F^n(Fx,z)
={(v\circ \widehat F^n)}_+(Fx).
\end{align*}
By $F$-invariance of $\nu$,
\[
	v_{n+1}^+=\int_\Delta {(v\circ \widehat F^{n+1})}_+\,d\nu\le 
	\int_\Delta {(v\circ \widehat F^n)}_+\circ F\,d\nu
	=\int_\Delta {(v\circ \widehat F^n)}_+\,d\nu=v_n^+.
\]
Hence $v_n^+$ is a monotone decreasing sequence bounded below by $-{|v|}_\infty$, and consequently 
 $\lim_{n\to\infty}v_n^+$ exists.  Similarly
$\lim_{n\to\infty}v_n^-$ exists.

Next, using uniform continuity of $v$ and the fact that $\diam\, \widehat F^n\pi^{-1}(x)\to0$ as $n\to\infty$ for each $x\in\Delta$,
\begin{align*}
	{(v\circ \widehat F^n)}_+(x)-{(v\circ \widehat F^n)}_-(x)
& =\sup_z v\circ \widehat F^n(x,z)- \inf_z v\circ \widehat F^n(x,z)
\\ & \le \sup_{z\in \widehat{F}^n\pi^{-1}(x)}v - \inf_{z\in \widehat{F}^n\pi^{-1}(x)}v \to0.
\end{align*}
Hence $\lim_{n\to\infty}v_n^+= \lim_{n\to\infty}v_n^-$.

Since $\int_{\widehat{\Delta}}v\,d\eta=\lim_{n\to\infty}\int_\Delta {(v\circ \widehat F^n)}_+\,d\nu$
we know  $\int_{\widehat{\Delta}}(v_1+v_2)\,d\eta\le \int_{\widehat{\Delta}}v_1\,d\eta+\int_{\widehat{\Delta}}v_2\,d\eta$.  Similarly, using
$\int_{\widehat{\Delta}}v\,d\eta=\lim_{n\to\infty}\int_\Delta {(v\circ \widehat F^n)}_-\,d\nu$
it follows that $\int_{\widehat{\Delta}}(v_1+v_2)\,d\eta\ge \int_{\widehat{\Delta}}v_1\,d\eta+\int_{\widehat{\Delta}}v_2\,d\eta$.  Hence $v\mapsto \int_{\widehat{\Delta}}v\,d\eta$ defines a linear functional on the space of continuous functions.  Clearly $\int_{\widehat{\Delta}}v\,d\eta\ge0$ whenever $v\ge0$, and $\int_{\widehat{\Delta}} 1\,d\eta = 1$, so $\eta$ is a probability measure.
Moreover, $\widehat F$-invariance of $\eta$ is immediate from the definition
$\int_{\widehat{\Delta}}v\,d\eta=\lim_{n\to\infty}\int_\Delta {(v\circ \widehat F^n)}_+\,d\nu$.
Finally, the fact that $\pi_*\eta=\nu$ is immediate from the definitions and the
invariance of $\nu$.
\end{proof}

\begin{rmk}
(a) In~\cite[Corollary~6.4]{APPV09}, it is shown that ergodicity of $\nu$
implies ergodicity of $\eta$.

\noindent (b)  
Given an ergodic $F$-invariant probability measure $\nu$, 
Proposition~\ref{prop:mux} shows how to construct an ergodic $\widehat F$-invariant measure $\eta$ with $\pi_*\eta=\nu$.    

Conversely, suppose that we are given an ergodic $\widehat F$-invariant probability measure $\eta_0$.
Then $\nu=\pi_*\eta_0$ is an
an ergodic $F$-invariant probability measure and gives rise via Proposition~\ref{prop:mux} to an 
ergodic $\widehat F$-invariant probability measure $\eta$.   We claim that $\eta=\eta_0$.

Indeed, suppose that $\eta_1$, $\eta_2$ are two 
ergodic $\widehat F$-invariant probability measures such that 
$\pi_*\eta_1=\pi_*\eta_2=\nu$.  We show that $\eta_1=\eta_2$.\footnote{We are grateful to Vitor Ara\'ujo for pointing out this argument.}
Let $v:\widehat\Delta\to\R$ be continuous and
define $S_n=n^{-1}\sum_{j=0}^{n-1}v\circ\widehat F^j$.
By the ergodic theorem, $\lim_{n\to\infty}S_n=\int v\,d\eta_i$ on a set $\widehat E_i\subset\widehat\Delta$ with $\eta_i(\widehat E_i)=1$ for $i=1,2$.
The proof of Proposition~\ref{prop:mux} shows that $\lim_{n\to\infty}n^{-1}\sum_{j=0}^{n-1}v\circ \widehat F^j(x,z)$ is independent of $z$ so
for $i=1,2$, there exist sets $E_i\subset\Delta$ with $\nu(E_i)=1$ such that
$\pi^{-1}E_i=\widehat E_i$.
In particular, $\widehat E_1\cap \widehat E_2\neq\emptyset$, and hence $\int v\,d\eta_1=\int v\,d\eta_2$.  Since $v$ is an arbitrary continuous function, $\eta_1=\eta_2$ as required.
\end{rmk}
  
\section{Disintegration}
\label{sec:disint}

In this section, we assume the same set up as in Section~\ref{sec:eta}.
Let $\cU:L^1(\Delta)\to L^1(\Delta)$ denote the Koopman operator
$\cU w=w\circ F$ corresponding to $F:\Delta\to\Delta$.
Define the transfer operator
$\cL:L^1(\Delta)\to L^1(\Delta)$ given by $\int_\Delta \cU w\, v \,d\nu=\int_\Delta w\,\cL v\,d\nu$ where $v\in L^1(\Delta)$, $w\in L^\infty(\Delta)$.

Let $0$ denote a distinguished point in $N$.
Following~\cite[Proposition~4.10 and Remark~4.11]{AV12}, we define $\eta_x$ almost everywhere as the limit as $n\to\infty$ of
$(\cL^nv_n)(x)$ where
$v_n(x)=v\circ \widehat F^n(x,0)$.
We note that the  argument below is considerably more direct and general than the one in~\cite[Section 4.4]{AV12}.  

\begin{prop}
\label{prop:etax}
For almost every $x\in\Delta$, the limit
 \[
	 \eta_x(v)=\lim_{n\to\infty} (\cL^nv_n)(x),\quad v_n(x)=v\circ \widehat F^n(x,0),
	 \]
	 exists for 
 every $v\in\cC^{0}(\widehat{\Delta})$ and 
	 defines a probability measure supported on $\pi^{-1}(x)$.

	 Moreover, for each $v\in\cC^0(\widehat\Delta)$, the map $x\mapsto \eta_{x}(v)$ lies in $L^\infty(\Delta)$ and 
\begin{equation}
\label{eq:defeta}
\eta(v) = \int_{\Delta} \eta_{x}(v) \ d\nu(x).
\end{equation}
\end{prop}

\begin{proof}
	First we consider a fixed $v\in\cC^0(\widehat\Delta)$.
Since $\cL\cU=I$, we have
\[
\cL^nv_n-\cL^{n+m}v_{n+m}=\cL^{n+m}(\cU^mv_n-v_{n+m}).
\]
Now
\[
	(\cU^mv_n)(x)-v_{n+m}(x)=v\circ \widehat F^n(F^mx,0)-v\circ \widehat F^{n+m}(x,0).
\]
By contractivity of the stable foliation,
$\diam\,\widehat F^n\pi^{-1}(F^mx)\to0$ as $n\to\infty$ uniformly in $m$ and $x$.
Hence by uniform continuity of $v$,
\begin{equation}
 \label{eq:C0}
 	\abs{\cU^mv_n-v_{n+m}}_\infty \to0
\end{equation}
as $n\to\infty$ uniformly in $m$.

Since $\nu$ is $F$-invariant, it follows from the duality definition of $\cL$ that
$|\cL v|_\infty \le |v|_\infty$ for all $v\in L^\infty(\Delta)$.
Hence
\[
	\abs{\cL^nv_n-\cL^{n+m}v_{n+m}}_\infty\le 
\abs{\cU^mv_n-v_{n+m}}_\infty\to0,
\]
as $n,m\to\infty$.
That is, $\cL^nv_n$ defines a Cauchy sequence in $L^\infty(\Delta)$.  In particular, the limit $\eta_x(v)$ exists for almost every $x$.
Note also that $|\eta_x(v)|\le |v|_\infty$.

It follows from separability of $\cC^0(\widehat\Delta)$ that 
the functional $v\mapsto \eta_x(v)$ defines a bounded linear functional on $\cC^0(\widehat{\Delta})$ for almost every $x\in \Delta$.
Moreover $\eta_x$ is positive and normalised and hence is identified with a probability measure on $\widehat{\Delta}$.
If $v|_{\pi^{-1}(x)}\equiv0$, then $(\cL^nv_n)(x)=0$ for all $n$ and so $\eta_x(v)=0$.  
Hence $\eta_x$ is supported on $\pi^{-1}(x)$.

Finally,
\[
	\int_\Delta\eta_x(v)\,d\nu(x)=\lim_{n\to\infty}\int_\Delta \cL^nv_n\,d\nu=\lim_{n\to\infty}\int_\Delta v\circ \widehat F^n(x,0)\,d\nu(x).
\]
Hence
\begin{align*}
	\eta(v)-
	\int_\Delta\eta_x(v)\,d\nu(x) & =\lim_{n\to\infty}\int_\Delta \Bigl((v\circ \widehat F^n)_+(x)-v\circ \widehat F^n(x,0)\Bigr)\,d\nu(x),
\end{align*}
which again converges to zero,
so $\eta(v)=\int_\Delta\eta_x(v)\,d\nu(x)$.
\end{proof}

\begin{rmk}
It follows that property 2 and the first part of property 3 
in~\cite[Definition~2.5]{AGY06} are automatically satisfied.
\end{rmk}

\section{H\"older regularity}
\label{sec:holder}

In this section, we continue to assume the set up in Sections~\ref{sec:eta} and~\ref{sec:disint}.   
In addition we suppose that $\Delta$ is a Riemannian manifold (with boundary), that
$F:\Delta\to\Delta$ is a $\cC^{1+\alpha}$ uniformly expanding map, as defined below, for some $\alpha\in(0,1]$ with absolutely continuous invariant probability measure $\nu$, and that $G$ is Lipschitz.
(Normally this situation would arise when there is a $\cC^{1+\alpha}$ stable foliation, in which case
we would have also that $G$ is $\cC^{1+\alpha}$, but we do not make explicit use of this extra structure.)
In the case $\alpha=1$, $\cC^{1+\alpha}$ means $\cC^{1+\Lip}$.

We write $\|x-x'\|$ and $\|z-z'\|$ for distance on
$\Delta$ and $N$.

\begin{defn} \label{def:F}   Let $\alpha\in(0,1]$.
	The map $F:\Delta\to\Delta$ is {\em uniformly expanding} if there is an open and dense subset $\Delta_0\subset\Delta$ with an at most countable partition into open sets $U_i$
	such that $F|_{U_i}:U_i\to \Delta_0$ is a $C^{1+\alpha}$ diffeomorphism onto $\Delta_0$ and extends to a homeomorphism from $\bar U_i$ onto $\Delta$ for each $i$.
Moreover, 
let $\cH_n$ denote the set of inverse branches of $F^n$ and write
$J_h := |\det(Dh)|$.   We require that there
 exist constants $C_J$, $C_{\lambda}$, $\lambda>0$ such that
  \begin{equation}
  \label{eq:exp}
  \abs{Dh(x)} \leq C_{\lambda}e^{-\lambda n}, \quad
 |\log J_h(x)-\log J_h(x')|/\norm{x-x'}^\alpha \le C_J,
  \end{equation}
 for all $h \in \cH_n, n\in \bN$, $x,x'\in\Delta_0$, $x\neq x'$.
\end{defn}

Let $d$ be a further metric on $\Delta$ with the property that
there is a constant $C_1>0$ such that $\norm{x-x'}\le C_1d(x,x')$ for all $x,x'\in\Delta$.
Write $\widehat F^n(x,z)=(F^nx,G_n(x,z))$.
  Set 
$\norm{v}_{\cB_\alpha(\Delta)}=\abs{v}_\infty+\abs{v}_{\cB_\alpha(\Delta)}$ where 
$\abs{v}_\infty=\sup_{x\in\Delta}\abs{v(x)}$ and
$\abs{v}_{\cB_\alpha(\Delta)}=\sup_{x,x'\in\Delta_0:x\neq x'}|v(x)-v(x')|/d(x,x')^\alpha$.
	Define $\cB_\alpha(\Delta)$ to be the Banach space of functions $v:\Delta\to\R$ with $\norm{v}_{\cB_\alpha(\Delta)}<\infty$.
	Similarly,
	define $\cB_\alpha(\widehat\Delta)$ to be the space of functions $v:\widehat\Delta\to\R$ with $\norm{v}_{\cB_\alpha(\widehat\Delta)}=\abs{v}_\infty+\abs{v}_{\cB_\alpha(\widehat\Delta)}<\infty$,
	where 
$\abs{v}_\infty=\sup_{x\in\widehat\Delta}\abs{v(x)}$ and
\[
	|v|_{\cB_\alpha(\widehat\Delta)}=\sup_{\stackrel{(x,z),(x',z')\in\widehat\Delta_0}{(x,z)\neq(x',z')}}\frac{|v(x,z)-v(x',z')|}{(d(x,x')+\norm{z-z'})^\alpha}.
\]
Note that $\norm{vw}_{\cB_\alpha(\Delta)}\le \norm{v}_{\cB_\alpha(\Delta)}\norm{w}_{\cB_\alpha(\Delta)}$ for all $v,w\in \cB_\alpha(\Delta)$ and similarly on $\widehat\Delta$.

If $d(x,x')=\|x-x'\|$, then $\cB_\alpha(\Delta)=\cC^\alpha(\Delta)$.
In this case we write $\abs{v}_\alpha=\abs{v}_{\cB_\alpha(\Delta)}$
and $\norm{v}_\alpha=\norm{v}_{\cB_\alpha(\Delta)}$.
In general 
$\cB_\alpha(\Delta)\supset\cC^\alpha(\Delta)$ and
similarly $\cB_\alpha(\widehat\Delta)\supset\cC^\alpha(\widehat\Delta)$.
Hence, the formulation allows for larger spaces of functions, including those which are Lipschitz with respect to a symbolic metric.

 A standard consequence of Definition~\ref{def:F} is the existence of a constant $C_J'$ such that
  \begin{equation}
  \label{eq:exp'}
  \sum_{h\in\cH_n}\norm{J_h}_\alpha \le C_J',
  \end{equation}
 for all $n\in \bN$.
We require in addition that there exists $n_0\ge1$ such that 
 \begin{align} \label{eq:star}
 \norm{G_{n_0}(x,z)-G_{n_0}(x,z')}\le \norm{z-z'},
 \end{align}
 for all $(x,z),(x,z')\in\widehat\Delta$.
Under the above assumptions we prove:

\begin{prop}
\label{prop:holder}
The disintegration ${\{\eta_{x}\}}_{x\in \Delta}$ is H\"older in the following  sense: there exists $C>0$ such that
for any $v\in \cB_\alpha(\widehat\Delta)$, the function $x\mapsto \bar{v}(x) :=  \eta_x(v)$ lies in $\cB_\alpha(\Delta)$ and
$\norm{\bar{v}}_{\cB_\alpha(\Delta)} \leq C \norm{v}_{\cB_\alpha(\widehat\Delta)}$. 
\end{prop}

\noindent For a bounded variation version of this result, see~\cite[Lemma A.7]{GP}.

To prove Proposition~\ref{prop:holder}, we require the following lemma.
  
 \begin{lemma}
 \label{lem:D_uG}
 There exists $C>0$ such that,
 for all $h\in \cH_n$, $n\in\bN$, $(x,z),\,(x',z)\in \widehat\Delta_0$,
 \begin{align*}
	 & \norm{G_n(hx,z)-G_n(hx',z)} \leq Cd(x,x'). 
 \end{align*}
 \end{lemma}

 \begin{proof}
Fix $n\in \bN$, $h\in \cH_n$. 
Let $n_0$ be as in~\eqref{eq:star}.
Since $G_m(x,z) = G_{n_0}(F^{m-n_0}x,G_{m-n_0}(x,z))$
for any $n_0\le m \leq n$, we have
$G_m(hx,z)=G_{n_0}(\ell x,G_{m-n_0}(hx,z))$
 where $\ell := F^{m-n_0}\circ h \in \cH_{n-m+n_0}$.
Hence
 \begin{multline*}
	 \|G_m(hx,z)-G_m(hx',z)\|
	 \\
  \le 
	 \|G_{n_0}(\ell x,G_{m-n_0}(hx,z))- G_{n_0}(\ell x',G_{m-n_0}(hx,z))\|
	 \\
 \quad +
	 \|G_{n_0}(\ell x',G_{m-n_0}(hx,z))- G_{n_0}(\ell x',G_{m-n_0}(hx',z))\|.
 \end{multline*}
 Using the estimates \eqref{eq:exp}, \eqref{eq:star} and the
 assumption on $d$,
 \begin{align*}
	 A_m  &\leq 
  \Lip\, G_{n_0} \|\ell x-\ell x'\|
  + A_{m-n_0} \\
  & \le \Lip\, G_{n_0} C_1C_{\lambda} e^{-\lambda(n-m+n_0)}d(x,x')
  + A_{m-n_0},
  \end{align*}
 where $A_m=\norm{G_m(hx,z)-G_m(hx',z)}$.
 Write $n=kn_0+r$ where $k\in\bN$ and $0\le r\le n_0-1$
 and set $m=jn_0+r$ where $j\le k$.  Then
 \[
	 A_{jn_0+r}
  \leq 
  \Lip\, G_{n_0} C_1C_{\lambda} e^{-\lambda(k-j+1)n_0}d(x,x')
  +A_{(j-1)n_0+r}.
 \]
 Consequently, iterating the above estimate, we obtain
\[
\begin{aligned}
	A_{kn_0+r}
  &\leq 
  \Lip\, G_{n_0}C_1C_{\lambda} \sum_{j=1}^n e^{-\lambda jn_0}d(x,x')+A_r\\
  &\leq \Lip\, G_{n_0}C_1C_{\lambda} (e^{\lambda n_0}-1)^{-1}d(x,x')+A_r.
\end{aligned}
  \]
The result follows since $\max_{r<n_0}A_r\le C_1C_\lambda \max_{r<n_0}\Lip\,G_r
\,d(x,x')\le Cd(x,x')$.
\end{proof}

  Recall that $\bar v=\lim_{n\to\infty}\cL^nv_n$ where $v_n(x)=v\circ\widehat F^n(x,0)$.
  Since $F$ is uniformly expanding, 
	  $\cL^nv_n=\varphi^{-1}\sum_{h\in\cH_n}J_h \ (\varphi v_n)\circ h$
where the density $\varphi$ corresponding to $\nu$ is 
$\cC^\alpha$ and bounded below.

 \begin{cor} \label{cor:Gn}
	 There exists $C>0$ such that
	 $\abs{(\varphi v_n)\circ h}_{\cB_\alpha(\Delta)}\le C\norm{v}_{\cB_\alpha(\widehat\Delta)}$,
	 for all $v\in\cB_\alpha(\widehat\Delta)$, $h\in\cH_n$, $n\in\bN$.
 \end{cor}

 \begin{proof}  Let $x,x'\in\Delta_0$.
Since $\widehat F^n(hx,0)=(x,G_n(hx,0))$, 
we have that $v_n\circ h(x)=v(x,G_n(hx,0))$.  Hence
\[
|v_n\circ h(x)-v_n\circ h(x')|\le|v|_{\cB_\alpha(\widehat\Delta)}(d(x,x')+\|G_n(hx,0)-G_n(hx',0)\|)^\alpha.
\]
	  Hence by Lemma~\ref{lem:D_uG}, there is a constant $C\ge1$ such that
\[
|v_n\circ h(x)-v_n\circ h(x')|\le C\abs{v}_{\cB_\alpha(\widehat\Delta)}d(x,x')^\alpha.
\]
Clearly $\abs{v_n\circ h}_\infty\le \abs{v}_\infty$,
so $\norm{v_n\circ h}_{\cB_\alpha(\Delta)}\le C\norm{v}_{\cB_\alpha(\widehat\Delta)}$.

	 Also, $\abs{\varphi\circ h}_\infty\le \abs{\varphi}_\infty<\infty$ and
	 \begin{align*}
		 |\varphi\circ h(x)-\varphi\circ h(x')|
& \le |\varphi|_\alpha\|hx-hx'\|^\alpha
		 \le \abs{\varphi}_\alpha \sup_{\xi\in\Delta_0}\abs{Dh(\xi)}^\alpha\norm{x-x'}^\alpha \\ &
		 \le C_1^\alpha C_\lambda^\alpha \abs{\varphi}_\alpha d(x,x')^\alpha,
	 \end{align*}
 so that $\norm{\varphi\circ h}_{\cB_\alpha(\Delta)}\le C$.
	 Finally,
	 \begin{align*}
		 \norm{(\varphi v_n)\circ h}_{\cB_\alpha(\Delta)} 
		  & \le \norm{\varphi\circ h}_{\cB_\alpha(\Delta)}\norm{v_n\circ h}_{\cB_\alpha(\Delta)}  
 \le C\norm{\varphi\circ h}_{\cB_\alpha(\Delta)}\norm{v}_{\cB_\alpha(\widehat\Delta)}
	 \end{align*}
	 as required.
  \end{proof}

	  \begin{lemma} \label{lem:L}
		  There exists $C>0$ such that
		  $\norm{\cL^nv_n}_{\cB_\alpha(\Delta)}\le C\norm{v}_{\cB_\alpha(\widehat\Delta)}$, for all $v\in\cB_\alpha(\widehat\Delta)$, $n\in\bN$.
	  \end{lemma}

	  \begin{proof}
		  It follows from the assumption on the metric $d$ that
$\norm{J_h}_{\cB_\alpha(\Delta)}\le C_1^\alpha\norm{J_h}_\alpha$.
	  Hence by Corollary~\ref{cor:Gn},
	  \[
\norm{J_h \ (\varphi v_n)\circ h}_{\cB_\alpha(\Delta)}
\le \norm{J_h}_{\cB_\alpha(\Delta)}
\norm{(\varphi v_n)\circ h}_{\cB_\alpha(\Delta)}
\le C\norm{J_h}_\alpha\norm{v}_{\cB_\alpha(\widehat\Delta)}.
\]
By estimate~\eqref{eq:exp'},
\[
	\Bigl\|
	\sum_{h\in \cH_n}  J_h\ (\varphi v_n)\circ h \Bigr\|_{\cB_\alpha(\Delta)}
	\le 
	C\|v\|_{\cB_\alpha(\widehat\Delta)}.
	\]
	Finally,
	\[
		\norm{\cL^nv_n}_{\cB_\alpha(\Delta)} \le 
		\norm{\varphi^{-1}}_{\cB_\alpha(\Delta)} \Bigl\|\sum_{h\in \cH_n}  J_h \ (\varphi v_n)\circ h\Bigr\|_{\cB_\alpha(\Delta)} 
		\le C\norm{\varphi^{-1}}_{\cB_\alpha(\Delta)} \norm{v}_{\cB_\alpha(\Delta)} 
\]
as required.
  \end{proof}

  It is elementary that if $f_n:\Delta\to\R$ is a sequence of H\"older functions with $\sup_n|f_n|_{\cB_\alpha(\Delta)}<\infty$ and $f_n\to f$ pointwise, 
  then $f\in\cB_\alpha(\Delta)$ and $\|f\|_{\cB_\alpha(\Delta)} \le \sup_n \|f_n\|_{\cB_\alpha(\Delta)}$.
  Hence
  Proposition~\ref{prop:holder} follows from Lemma~\ref{lem:L} by setting
  $f_n=\cL^nv_n$ and $f=\bar v$.

\subsection*{H\"older disintegration for suspensions}
The following generalization to suspensions turns out to be useful in~\cite{AMV14}.
Let $R:\widehat\Delta\to \R^+$ be a measurable roof function that is constant along stable leaves, so that $R:\Delta\to\R^+$ is well-defined.
Form the suspensions 
$\Delta^R=\{(x,u)\in\Delta\times\R:0\le u\le R(x)\}$ and
$\widehat\Delta^R=\{(x,z,u)\in\widehat\Delta\times\R:0\le u\le R(x)\}$.

For $v:\Delta^R\to\R$, we define $\|v\|_{\cB_\alpha(\Delta^R)}=
|v|_\infty+|v|_{\cB_\alpha(\Delta^R)}$ where
\[
|v|_{\cB_\alpha(\Delta^R)}=\sup_{\stackrel{(x,u),(x',u)\in\Delta_0^R}{x\neq x'}}\frac{|v(x,u)-v(x',u)|}{d(x,x')^\alpha}.
\]
Similarly, for $v:\widehat\Delta^R\to\R$, we define $\|v\|_{\cB_\alpha(\widehat\Delta^R)}=
|v|_\infty+|v|_{\cB_\alpha(\widehat\Delta^R)}$ where
\[
	|v|_{\cB_\alpha(\widehat\Delta^R)}=\sup_{\stackrel{(x,z,u),(x',z',u)\in\widehat\Delta_0^R}{(x,z)\neq(x',z')}}\frac{|v(x,z,u)-v(x',z',u)|}{(d(x,x')+\norm{z-z'})^\alpha}.
\]
Let $\cB_\alpha(\Delta^R)$ 
and $\cB_\alpha(\widehat\Delta^R)$ denote the corresponding spaces of continuous observables
for which $\|v\|_{\cB_\alpha(\Delta^R)}$ and
$\|v\|_{\cB_\alpha(\widehat\Delta^R)}$ respectively are finite.

Suppose that $v:\widehat\Delta^R\to\R$.  Write $v^u(x,z)=v(x,z,u)$ and note that for fixed $u\ge0$, the function $v^u$ is defined on the set 
$\bigcup_{(x,u)\in\Delta^R}\pi^{-1}(x)$.
Hence we can define $\eta_x(v^u)$ whenever $(x,u)\in\Delta^R$.
In this way, we obtain a function $\bar v:\Delta^R\to\R$ given by
\[
	\bar v(x,u)=\eta_x(v^u).
\]

\begin{prop}
\label{prop:holderflow}
There exists $C>0$ such that
for any $v\in \cB_\alpha(\widehat\Delta^R)$, the function $(x,u)\mapsto \bar{v}(x,u) =  \eta_x(v^u)$ lies in $\cB_\alpha(\Delta^R)$ and
$\norm{\bar{v}}_{\cB_\alpha(\Delta^R)} \leq C \norm{v}_{\cB_\alpha(\widehat\Delta^R)}$. 
\end{prop}

\begin{proof}
	This is proved in the same way as Proposition~\ref{prop:holder},
	but care needs to be taken with the notation since $v^u$ is not well-defined on the whole of $\hat\Delta$.

For fixed $u$, choose a continuous extension $w:\hat\Delta\to\R$ of $v^u$.
Then for $(x,u)\in\Delta^R$, we have
\[
\bar v(x,u)=\eta_x(w)=
\lim_{n\to\infty}(\cL^nw_n)(x), \quad w_n(x)=w\circ\widehat F^n(x,0).
\]

But
$\cL^nw_n =\varphi^{-1}\sum_{h\in\cH_n}J_h\,(\varphi w_n)\circ h$,
and
$w_n\circ h(x)=w\circ \hat F^n(hx,0)=w(x,G_n(hx,0))=
v^u(x,G_n(hx,0))= v(x,G_n(hx,0),u)$.
Hence for $(x,u)\in\Delta^R$, we have shown that
\[
	\bar v(x,u)=\lim_{n\to\infty}(M_nv)(x,u), \quad
	M_nv=	\tilde\varphi^{-1}\sum_{h\in\cH_n}\tilde J_h\,\tilde\varphi \circ\tilde h\, \tilde v_n,
\]
where 
\[
	\tilde h(x,u)=(hx,u), \quad
	\tilde J_h(x,u)=J_h(x), \quad
	\tilde\varphi(x,u)=\varphi(x), \quad
	\tilde v_n(x,u)=v(x,G_n(hx,0),u).
\]
It now suffices to prove that
${\|M_nv\|}_{\cB_\alpha(\Delta^R)}\le C\|v\|_{\cB_\alpha(\widehat\Delta^R)}$.

The main steps can now be sketched as follows.
	Picking up at the beginning of the proof of Corollary~\ref{cor:Gn},
for $(x,u),\,(x',u)\in \Delta_0^R$,
	\begin{align*}
		|\tilde v_n(x,u)-\tilde v_n(x',u)| & 
\le|v|_{\cB_\alpha(\widehat\Delta^R)}(d(x,x')+\|G_n(hx,0)-G_n(hx',0)\|)^\alpha
\\  &\le C|v|_{\cB_\alpha(\widehat\Delta^R)}d(x,x')^\alpha,
\end{align*}
and we deduce that $\|\tilde v_n\|_{\cB_\alpha(\Delta^R)}
\le C{\|v\|}_{\cB_\alpha(\widehat\Delta^R)}$.

Next, it follows as before that 
${\|\tilde\varphi\circ\tilde h\|}_{\cB_\alpha(\Delta^R)}\le C$ so that
$\|(\tilde\varphi \tilde v_n)\circ\tilde h\|_{\cB_\alpha(\Delta^R)}
\le C\|v\|_{\cB_\alpha(\widehat\Delta^R)}$.

Turning to
Lemma~\ref{lem:L},
the estimate ${\|\tilde J_h\|}_{\cB_\alpha(\Delta^R)}\le C{\|J_h\|}_\alpha$
holds just as before, leading to the desired estimate
${\|M_nv\|}_{\cB_\alpha(\Delta^R)}\le C{\|v\|}_{\cB_\alpha(\widehat\Delta^R)}$.
\end{proof}

\section{$\cC^{1}$ regularity}
\label{sec:smooth}

As in the previous section, we assume the set up in Sections~\ref{sec:eta} and~\ref{sec:disint}. Now we require yet more regularity for the system; namely that $N$ is a compact manifold possibly with boundary, that $G:\Delta\times N\to N$ is $\cC^{1}$ and that
$F:\Delta\to\Delta$ is a $\cC^{2}$ uniformly expanding map (as in Definition~\ref{def:F} but with $C^{1+\alpha}$ changed to $C^2$) with absolutely continuous invariant probability measure $\nu$.
As before $\cH_n$ denotes the set of inverse branches of $F^n$, each defined on the open and dense subset $\Delta_0$ of $\Delta$,
and $J_h := |\det(Dh)|$, and we require that there exists $C_\lambda>0$, $\lambda>0$ such that
  \begin{equation}
  \label{eq:Ch}
  \sup_{x\in\Delta_0}\abs{Dh(x)}  \leq C_{\lambda}e^{-\lambda n},
  \end{equation}
 for all $h \in \cH_n, n\in \bN$.
Let $\widehat\Delta_0=\Delta_0\times N$.

 In the following we use the notation $Dv = (D_{u}v, D_{s}v)$ and $DG = (D_uG,D_sG)$. 
We require in addition the following uniform exponential contraction in the stable direction:
 the constants $C_\lambda>0$, $\lambda>0$ can be chosen so that also
  \begin{equation}
  \label{eq:CG}
  \norm{D_{s}G_n(x,z)} \leq C_\lambda e^{-\lambda n}, 
  \end{equation}
 for all $n\in \bN$, $(x,z) \in \widehat\Delta_0$.

Define $\cC^1(\Delta)$ to be the space of continuous functions $v:\Delta\to\R$ that are continuously differentiable on $\Delta_0$ with bounded derivative.  This is a Banach space under the norm 
$\norm{v}_{\cC^1}=\sup_{x\in\Delta}|v(x)|+
\sup_{x\in\Delta_0}|Dv(x)|$.
The space $\cC^1(\widehat\Delta)$ is defined similarly.

Under these assumptions we prove:
\begin{prop}
\label{prop:etaxC1}
The disintegration ${\{\eta_{x}\}}_{x\in \Delta}$ is smooth in the following  sense:
there exists $C>0$ such that, 
for any $v\in \cC^1(\widehat \Delta)$, the function $x\mapsto \bar{v}(x) :=  \eta_x(v)$ lies in $\cC^{1}(\Delta)$ and
\[
	\norm{D\bar{v}(x)} \leq C \sup_{z\in N}\abs{v(x,z)}+ 
	C \sup_{z\in N}\norm{Dv(x,z)},
\]
for all $x\in\Delta_0$.
\end{prop}

\noindent This fills a gap in~\cite{AV12} since there are inaccuracies in 
terms (30) and (31) therein.
\begin{rmk}
 The estimate of Proposition~\ref{prop:etaxC1} corresponds to property~(3) of~\cite[Definition 2.5]{AGY06}. 
Here we have an additional term, but the application of the estimate in~\cite[\S8]{AGY06} is unaffected. 
\end{rmk}

The remainder of this section is devoted to the proof of Proposition~\ref{prop:etaxC1}.

 For all $n\in \bN$ and $v\in\cC^0(\widehat{\Delta})$, let
 \[
	 M_nv(x) := (\cL^nv_n)(x)=
\sum_{h\in \cH_n} \big( \tfrac{\varphi \circ h}{\varphi} \cdot J_{h}  \big)(x)
\cdot v (x, G_n(hx,0)). 
\]
By Proposition~\ref{prop:etax}, 
$M_nv$ converges in $\cC^{0}(\Delta)$ and $M_nv(x)\to\eta_x(v)$.

We first show that $M_nv$ is Cauchy in $\cC^1(\Delta)$.
Recall that
$M_nv - M_{n+m}v =\cL^{n+m}K_{n,m}$ where
$K_{n,m}=\cU^mv_n-v_{n+m}$.  Hence
 \begin{equation}
 \label{eq:MnMm}
 M_nv - M_{n+m}v = \sum_{\ell\in \cH_{n+m}} \Big( \tfrac{\varphi \circ \ell}{\varphi} \cdot J_{\ell}  \Big) \cdot K_{n,m}\circ\ell.
 \end{equation}
 We note that
 $ K_{n,m}(\ell x) =
 v(x, G_n(F^m\circ \ell(x),0)) - v(x, G_{n+m}( \ell x,0))$.
 
 The compact manifold $N$ can be smoothly embedded as a submanifold of a vector space $\R^d$.  We fix such an embedding, so the quantity $D_uG_n(hx,z)-D_uG_n(hx,z')$ below is well-defined.

 \begin{lemma}
 \label{lem:estD_uG}
 (a)
 There is a constant $C>0$ such that
 for all $m,n\in \bN$ with $m\le n$, and all $h\in \cH_n$, $(x,z)\in \widehat\Delta_0$,
 \begin{align}
 \label{eq:estD_uG}
 & \norm{D_uG_m(hx,z) \ Dh(x)} \leq Ce^{-\lambda(n-m)}. \end{align}
 (b) There is a constant $C>0$ such that
 for all $n\in \bN$, $h\in \cH_n$, $(x,z),(x,z')\in \widehat\Delta_0$,
 \begin{align}
 \label{eq:estD_uG2}
 & \norm{[D_uG_n(hx,z)-D_uG_n(hx,z')] \ Dh(x)} \leq Ce^{-\lambda n}\norm{z-z'} .
 \end{align}
 \end{lemma}

 \begin{proof}
	 (a)	 Choose $n_0\ge1$ sufficiently large that $(1+C_\lambda) e^{-\lambda n_0}\le1$.  Let $C=C_\lambda\max_{j\le n_0}\norm{G_j}_{\cC^1}$.
	 Then for $m\le n_0$, and all $n\in\bN$, $h\in\cH_n$,
	 it follows from~\eqref{eq:Ch} that
	 \[
		 \norm{D_uG_m(hx,z) \ Dh(x)}\le \norm{G_m}_{\cC^1}C_\lambda e^{-\lambda n}\le Ce^{-\lambda(n-m)}.
	 \]

	 It remains to consider the case $m\ge n_0$.  We proceed by induction.
	 Let $h\in\cH_n$, $n\in\bN$.
Since $G_m(x,z) = G_{n_0}(F^{m-n_0}x,G_{m-n_0}(x,z))$, we have
$G_m(hx,z)=G_{n_0}(\ell x,G_{m-n_0}(hx,z))$
 where $\ell := F^{m-n_0}\circ h \in \cH_{n-m+n_0}$.
Hence
 \begin{align} \label{eq:a} 
  D_uG_m(hx,z) \ Dh(x)
  &=
  D_uG_{n_0}(\ell x, G_{m-n_0}(hx,z)) \ D\ell(x) \\
& \quad +
  D_sG_{n_0}(\ell x, G_{m-n_0}(hx,z)) \   D_uG_{m-n_0}(hx,z) \ Dh(x).
\nonumber\end{align}
Using the estimate \eqref{eq:Ch} and the induction hypothesis,
\begin{align*}
  \norm{  D_uG_m(hx,z) \ Dh(x)}
  & \leq 
  \norm{G_{n_0}}_{\cC^1} C_{\lambda} e^{-\lambda(n-m+n_0)}
  + C_\lambda e^{-\lambda n_0} C e^{-\lambda (n-m+n_0)}
  \\
  &
  \le C(1+C_\lambda)e^{-\lambda n_0}e^{-\lambda(n-m)}
  \le Ce^{-\lambda(n-m)},
  \end{align*}
completing the proof of part~(a).

\noindent (b) 
Choose $n_0\ge 1$ sufficiently large that $e^{-\lambda n_0}C_\lambda^2<1$.
Let $C_0=\max_{j\le n_0}\norm{G_j}_{\cC^2}$ and let $C_1$ be the constant in part (a).  Choose $C\ge C_0C_\lambda$ such that
	\[
		C_0C_\lambda^2 +C_0C_1C_\lambda + e^{-\lambda n_0}C_\lambda^2 C \le C.
	\]

For $n\le n_0$, $h\in\cH_n$, it follows from~\eqref{eq:Ch} that 
\[
	\norm{[D_uG_n(hx,z)-D_uG_n(hx,z')] \ Dh(x)} \leq C_0\norm{z-z'}C_\lambda e^{-\lambda n}\le Ce^{-\lambda n}\norm{z-z'},
\]
so it remains to consider the case $n\ge n_0$.

Starting from~\eqref{eq:a}, we have for all $m\ge n_0$,
\[
	[D_uG_m(hx,z)-D_uG_m(hx,z')] \ Dh(x) = I_1+I_2+I_3,
\]
where
\begin{align*}
	I_1 & = [D_uG_{n_0}(\ell x, G_{m-n_0}(hx,z))-D_uG_{n_0}(\ell x, G_{m-n_0}(hx,z'))] \ D\ell(x), \\
	I_2 & = [D_sG_{n_0}(\ell x, G_{m-n_0}(hx,z))-D_sG_{n_0}(\ell x, G_{m-n_0}(hx,z'))]
	\ D_uG_{m-n_0}(hx,z) \ Dh(x), \\
	I_3 & = D_sG_{n_0}(\ell x, G_{m-n_0}(hx,z')) \ 
	[D_uG_{m-n_0}(hx,z)-D_uG_{m-n_0}(hx,z')] \ Dh(x) .
\end{align*}
Using estimates~\eqref{eq:Ch},~\eqref{eq:CG} and part (a),
\begin{align*}
	\norm{I_1} 
	& \le  C_0\norm{G_{m-n_0}(hx,z)-G_{m-n_0}(hx,z')} C_\lambda e^{-\lambda(n-m+n_0)}
	 \le  C_0C_\lambda^2 e^{-\lambda n}\norm{z-z'}, \\
	\norm{I_2} & \le C_0\norm{G_{m-n_0}(hx,z)-G_{m-n_0}(hx,z')}
	C_1e^{-\lambda(n-m+n_0)}
	 \le C_0C_1C_\lambda e^{-\lambda n}\norm{z-z'}.
	\end{align*}

	Writing $h=k\circ\ell$ where $k\in \cH_{m-n_0}$, $\ell\in\cH_{n-m+n_0}$,
	\begin{align*}
		& [D_uG_{m-n_0}(hx,z)-D_uG_{m-n_0}(hx,z')] \ Dh(x)
		\\ & \qquad\qquad\qquad =
[D_uG_{m-n_0}(k(\ell x),z)-D_uG_{m-n_0}(k(\ell x),z')] \ Dk(\ell x) \ D\ell(x)
.
\end{align*}
It follows inductively that
\begin{align*}
		\| [D_uG_{m-n_0}(hx,z)-D_uG_{m-n_0}(hx,z')] \ Dh(x) \|
		& \le Ce^{-\lambda(m-n_0)} C_\lambda e^{-\lambda(n-m+n_0)}\norm{z-z'} \\ &
		= CC_\lambda e^{-\lambda n}\norm{z-z'}.
		\end{align*}
		Hence
\[
\norm{I_3}  \le C_\lambda^2e^{-\lambda n_0}C e^{-\lambda n}\norm{z-z'}.
\]
It follows from the choice of $n_0$ and $C$ that 
$\norm{[D_uG_m(hx,z)-D_uG_m(hx,z')] \ Dh(x)}\le Ce^{-\lambda n} \norm{z-z'}$ for all $m\ge n_0$, and the proof of (b) is complete.  
\end{proof}

 \begin{lemma}
 \label{lem:estDK}
Suppose $v\in \cC^{1}(\widehat\Delta)$ and that $K_{n,m}$ is defined as above.
Then
\[
	\sup_{m\in \bN}\sup_{\ell\in\cH_{n+m}} \sup_{x\in \Delta_0} \norm{ DK_{n,m}(\ell x)D\ell(x)}
\to 0,
\quad
\text{as $n \to \infty$}.
\] 
 \end{lemma}
 \begin{proof}
Recall that $ K_{n,m}(\ell x) =
 v(x, G_n(F^m\circ \ell(x),0)) - v(x, G_{n+m}( \ell x,0))$.
 Since $G_{n+m}( \ell x,0) = G_n( F^m\circ \ell(x),G_m(\ell x,0))$,
 \[
 K_{n,m}(\ell x)
  =
 v(x, G_n(hx,0)) - v(x, G_n( hx,G_m(\ell x,0))),
 \]
 where for convenience we write $h = F^m \circ \ell$. 
 Differentiating we obtain 
 \begin{equation}
 \label{eq:diffK}
	 DK_{n,m}(\ell x) D\ell(x) = J_1+J_2-J_3,
	\end{equation}
	where writing $z=G_m(\ell x,0)$,
\begin{align*}
  J_1 &= 
 D_{u}v(x,G_n(hx,0))-
  D_{u}v(x,G_n(hx,z)), \\
  J_2 & =
  D_{s}v(x,G_n(hx,0)) \ D_uG_n(hx,0) \ Dh(x) \\ & \qquad -
   D_{s}v(x,G_n(hx,z)) \
   D_uG_n(hx, z) \ Dh(x),\\
  J_3 &  = 
  D_{s}v(x,G_n(hx,z)) \
   D_sG_n(hx, z) \ D_{u}G_m(\ell x,0) \ D\ell(x).
 \end{align*}

 By~\eqref{eq:CG},
\[ 
	\norm{G_n(hx,0)- G_n(hx,z)}
\leq C_\lambda e^{-\lambda n}\diam(\pi^{-1}(x)).
 \]
 Therefore, by the uniform continuity of $Dv$, we have that $\norm{J_1}\to0$ uniformly in $x$, $\ell$ and $m$ as $n\to\infty$.

 Next, $J_2=J_2'+J_2''$ where
 \begin{align*}
  J_2' & =
 [D_{s}v(x,G_n(hx,0)) -D_{s}v(x,G_n(hx,z)) ]\ D_uG_n(hx,0) \ Dh(x),  \\
  J_2'' & =
   D_{s}v(x,G_n(hx,z)) \ [D_uG_n(hx,0) - D_uG_n(hx, z)] \ Dh(x).
 \end{align*}
 The same argument used for $J_1$ shows that
 \[
\sup_{x,\ell,m}  \norm{  D_{s}v(x,G_n(hx,0))  - D_{s}v(x,G_n(hx,z)) } \to 0
 \]
 as $n\to \infty$.
 Combining this with Lemma~\ref{lem:estD_uG}(a) we obtain
 that $\norm{J_2'}\to0$ uniformly in $x$, $\ell$ and $m$ as $n\to\infty$.

 The first factor of $J_2''$ is bounded by $\norm{v}_{\cC^1}$, so it follows from
 Lemma~\ref{lem:estD_uG}(b) that $\norm{J_2''}\to0$ uniformly in $x$, $\ell$ and $m$ as $n\to\infty$.

 Turning to $J_3$, the three factors are bounded by 
 $\norm{v}_{\cC^1}$, $C_\lambda e^{-\lambda n}$ and
 $Ce^{-\lambda(n+m-m)}=Ce^{-\lambda n}$ respectively, where we have used~\eqref{eq:CG} and
Lemma~\ref{lem:estD_uG}(a). 
Hence 
 $\norm{J_3}\to0$ uniformly in $x$, $\ell$ and $m$ as $n\to\infty$.
The combination of these estimates completes the proof of the lemma.
 \end{proof}

 A standard consequence of the assumptions used in this section is that there exists $C_{d}>0$ such that 
 \begin{equation}
 \label{eq:DJn}
  \sum_{h\in \cH_n} \norm{{D  J_{h}(x)}} \leq C_{d}
  \quad \quad
  \text{ for all $n\in \bN$, $x\in \Delta_0$.}
 \end{equation}
Observe that, differentiating \eqref{eq:MnMm},
 \[
	 D(M_nv - M_{n+m}v) 
  =
  \sum_{\ell\in \cH_{n+m}} D\Big( \frac{\varphi \circ \ell}{\varphi} \cdot J_{h}  \Big) \cdot K_{n,m}\circ\ell
   +
   \Big( \frac{\varphi \circ \ell}{\varphi} \cdot J_{h}  \Big) \cdot \Big(DK_{n,m}\Bigr)\circ\ell \ D\ell.
 \]
 Using  \eqref{eq:DJn} and  Lemma~\ref{lem:estDK}, together with the previously proven fact \eqref{eq:C0} that $\sup_{x,m}\norm{K_{n,m}(x)} \to 0$ as $n\to \infty$ and that $ \sum_{\ell\in \cH_n}   \Big( \frac{\varphi \circ \ell}{\varphi} \cdot J_{h}  \Big)(x) = 1$ proves that
 \[
	 \sup_{x\in\Delta_0} \sup_{m\in\bN} \norm{D(M_nv - M_{n+m}v)(x)  } \to 0
 \]
 as $n\to \infty$  and hence the sequence $M_nv$ is Cauchy in $\cC^{1}(\Delta)$.
 This proves the first claim of Proposition~\ref{prop:etaxC1}, namely that $\bar{v}\in\cC^{1}(\Delta)$.
 Moreover, we have shown that $M_nv\to\bar{v}$ in $\cC^1(\Delta)$ so it remains to show that there exists $C>0$ such that
 $\norm{M_nv}_{\cC^1}\le C\norm{v}_{\cC^1}$.
 It is clear that $\norm{M_nv}_{\cC^0}\le \norm{v}_{\cC^0}$, so it remains to prove: 
\begin{lemma}
\label{lem:estC1}
There exists $C>0$ such that,
for all $v\in \cC^1(\widehat \Delta)$, $n\in \bN$, $x\in\Delta_0$,
\[
 \norm{DM_nv(x)} \leq 
C \sup_{z\in N} \abs{v(x,z)}+
C \sup_{z\in N} \norm{Dv(x,z)}.
\]
\end{lemma}

\begin{proof}
 Write
 $  M_nv(x) = \sum_{h\in \cH_n} ( \frac{\varphi \circ h}{\varphi} \cdot J_{h}  )(x) \cdot B_n(x)$,
 where
$B_n(x) = v(x, G_n(hx,0))$. 
 First we estimate $\norm{ DB_n(x)}$.
 Differentiating we obtain
 \[
  DB_n(x) 
  = 
  D_{u}v(x,G_n(hx,0))
  + 
  D_{s}v(x,G_n(hx,0)) \ D_uG_n(hx,0) \ Dh(x)
 \]
 and so by Lemma~\ref{lem:estD_uG},
 \[
  \norm{DB_n(x) }
  \leq
 \sup_{z \in N} \norm{Dv(x,z)} (1+ \norm{D_uG_n(hx,0) \ Dh(x)}) \le
 C\sup_{z \in N} \norm{Dv(x,z)}.
 \]
Clearly, $\abs{B_n(x)}\le 
 \sup_{z \in N} \abs{v(x,z)}$.
 Using also the estimate from \eqref{eq:DJn} 
 completes the proof of the lemma.
\end{proof}


\end{document}